\documentclass[12pt,twoside,a4paper]{article}

\usepackage{amsmath,amsthm,amssymb,tikz-cd,hyperref}

\newtheorem{theorem}{Theorem}[section]					
\newtheorem{lemma}[theorem]{Lemma}
\newtheorem{lemmadefn}[theorem]{Lemma/Definition}

\theoremstyle{definition}
\newtheorem{definition}[theorem]{Definition}
\theoremstyle{definition}
\newtheorem{example}[theorem]{Example}
\newtheorem{nothing}[theorem]{}

\newcommand{\N}{\mathbb{N}}								
\newcommand{\Z}{\mathbb{Z}}

\newcommand{\K}{\mathbb{K}}
							
\newcommand{\OO}{\mathcal{O}}

\newcommand{\Aut}{\operatorname{Aut}}
\newcommand{\BE}{\mathcal{BE}}
\newcommand{\Bl}{\operatorname{Bl}}
\newcommand{\bli}{\operatorname{bli}}
\newcommand{\BP}{\mathcal{BP}}

\newcommand{\dsum}{\oplus}

\newcommand{\gp}[1]{\langle#1\rangle}
\newcommand{\Hom}{\operatorname{Hom}}

\newcommand{\id}{\operatorname{id}}

\newcommand{\Inn}{\operatorname{Inn}}

\newcommand{\Irr}{\operatorname{Irr}}
\newcommand{\iso}{\cong}

\newcommand{\isoto}{\overset{\sim}{\to}}

\newcommand{\nor}{\trianglelefteq}

\newcommand{\onto}{\twoheadrightarrow}

\newcommand{\Out}{\operatorname{Out}}

\newcommand{\Res}{\operatorname{Res}}

\newcommand{\set}[1]{\left\{#1\right\}}

\newcommand{\subgp}{\leq}
\newcommand{\Syl}{\operatorname{Syl}}

\newcommand{\tensor}{\otimes}

	\makeatletter
	\newcommand{\tpitchfork}{%
		\vbox{
			\baselineskip\z@skip
			\lineskip-.52ex
			\lineskiplimit\maxdimen
			\m@th
			\ialign{##\crcr\hidewidth\smash{$-$}\hidewidth\crcr$\pitchfork$\crcr}
		}%
	}
	\makeatother

\title{Isotypic blocks of finite group algebras that are not $p$-permutation equivalent}
\author{John Revere McHugh\\john.r.mchugh@du.edu}

\begin{document}
	
\maketitle

\begin{flushleft}
	Department of Mathematics\\
	University of Denver\\
	C. M. Knudson Hall\\
	2390 S. York St, Denver, CO
\end{flushleft}
	
\begin{abstract}
	We show that Kessar's isotypy between Galois conjugate blocks of finite group algebras does not always lift to a $p$-permutation equivalence. We also provide examples of Galois conjugate blocks which are isotypic but not $p$-permutation equivalent. These results help to clarify the distinction between a $p$-permutation equivalence and an isotypy, and may be useful in determining necessary and sufficient conditions for when an isotypy lifts to a $p$-permutation equivalence.
\end{abstract}

\section{Introduction}

Let $G$ and $H$ be finite groups, let $p$ be a prime number, and let $(\K,\OO,k)$ be a $p$-modular system large enough for $G$ and $H$ such that $k$ is algebraically closed. If $A$ is a block of $\OO G$ and $B$ is a block of $\OO H$ then there are many different notions of ``equivalence'' between $A$ and $B$. For example, there might exist a splendid Rickard equivalence between $A$ and $B$ \cite{Rickard_1996}, a $p$-permutation equivalence \cite{Boltje_2020}, or an isotypy \cite{Broue_1990}. Brou\'{e}'s abelian defect group conjecture states that each of these three block equivalences should exist between any block with abelian defect groups and its Brauer correspondent.

Boltje and Perepelitsky showed in \cite[Section 15]{Boltje_2020} that a splendid Rickard equivalence between blocks $A$ and $B$ induces a $p$-permutation equivalence, and that a $p$-permutation equivalence induces an isotypy between $A$ and $B$ (the latter construction is outlined in \ref{noth:isofrompperm} below). However, the authors did not determine whether every isotypy between $A$ and $B$ ``comes from'' a $p$-permutation equivalence. The purpose of this note is to provide an example of an isotypy that does not come from a $p$-permutation equivalence. These results might be considered as a first step in the overall goal of determining necessary and sufficient conditions for when an isotypy lifts to a $p$-permutation equivalence. Of course, if an isotypy does not lift to a $p$-permutation equivalence then it cannot lift to a splendid Rickard equivalence, either. Examples of blocks that are isotypic but not derived equivalent --- and hence not splendid Rickard equivalent --- have already been identified (\cite{Benson_2007}, \cite{Kessar_2012}).

The isotypies we examine exist between ``Galois conjugate'' blocks of a fixed finite group $G$ and were first constructed by Kessar \cite{Kessar_2012}. We show in Theorem \ref{thm:two} that if $a$ and $b$ are Galois conjugate block idempotents then the isotypy of \cite{Kessar_2012} between $a$ and $b$ does not always lift to a $p$-permutation equivalence. This still leaves the possibility that a $p$-permutation equivalence between $a$ and $b$ exists (which would have to induce some other isotypy between $a$ and $b$ besides the one defined in \cite{Kessar_2012}). However, we also show in Theorem \ref{thm:three} that it is possible that no $p$-permutation equivalence between $a$ and $b$ exists. A method for constructing particular examples of this phenomenon is given in Example \ref{ex:one}. 

The key ingredient in the proofs of our results is the use of \textit{K\"{u}lshammer-Puig classes}, which are cohomological invariants associated to a given block. A $p$-permutation equivalence ``preserves'' K\"{u}lshammer-Puig classes \cite[Theorem 13.4]{Boltje_2020}, while an isotypy need not. The failure of an isotypy to preserve these cohomology classes thus leads to the failure to lift the isotypy to a $p$-permutation equivalence in certain instances.

\section{Preliminaries}

Let $G$ be a finite group.

\begin{nothing}\textit{Group theoretic notation.}
	We write $C_G(P)$, $N_G(P)$ and $Z(P)$ for the centralizer, normalizer and center, respectively, of a subgroup $P\subgp G$. The largest normal $p$-subgroup of $G$ is denoted $O_p(G)$. The group of inner automorphisms of $G$ is $\Inn(G)$, and $\Out(G)=\Aut(G)/\Inn(G)$. 
	
	If $g\in G$ then $c_g:G\to G$ is the conjugation map $c_g(x)=gxg^{-1}={}^gx$, $x\in G$. We also write ${}^gP$ for the image of $P$ under $c_g$. 
	
	If $H$ is another finite group, $Q\subgp H$, $P\subgp G$, and $\phi:Q\isoto P$ is a group isomorphism, then
	\begin{equation*}
		\Delta(P,\phi,Q)=\set{(\phi(y),y)|y\in Q}
	\end{equation*}
	is a \textit{twisted diagonal} subgroup of $G\times H$. If $P=Q$ and $\phi=\id_Q$ then we write $\Delta(Q)$ in place of $\Delta(P,\phi,Q)$. 
	
	If $g\in G$ then $g_p,g_{p'}\in G$ denote the $p$-part and $p'$-part of $g$, respectively. Let $G_{p'}$ denote the set of $p'$-elements of $G$.
\end{nothing}

\begin{nothing}\label{noth:pmodsys}$p$\textit{-modular systems, the Frobenius automorphism, and the antipode.}
	Throughout, $p$ is a prime number and $(\K,\OO,k)$ is a $p$-modular system such that $k$ is algebraically closed of characteristic $p$. When working with a finite group $G$ we also assume that $(\K,\OO,k)$ is ``large enough for $G$,'' which here means that $\K$ contains a primitive $|G|$th root of unity. 
	
	Let $R$ denote either of the rings $\OO$ or $k$. The image of an element $\lambda\in\OO$ under the canonical surjection $\OO\onto k$ will be denoted $\overline{\lambda}$. For $\lambda\in k$ we also set $\overline{\lambda}=\lambda$. The map $\overline{\cdot}:R\onto k$ just defined extends linearly to an $R$-algebra homomorphism $\overline{\cdot}:RG\onto kG$. Of course, if $R=k$ this map is just the identity automorphism of $kG$.
	
	Let $\sigma:k\isoto k$ denote the Frobenius automorphism of $k$, which is defined by $\sigma(\lambda)=\lambda^p$ for all $\lambda\in k$. Then $\sigma$ extends linearly to a ring automorphism of $kG$ which we also denote by $\sigma$. Note that $\sigma$ is not a $k$-algebra homomorphism of $kG$.
	
	The antipode of $RG$ is the $R$-module automorphism $(\cdot)^\ast:RG\to RG$ defined by $g^\ast=g^{-1}$ for any $g\in G$. This map is an anti-involution, i.e., satisfies $(\alpha^\ast)^\ast=\alpha$ and $(\alpha\beta)^\ast=\beta^\ast\alpha^\ast$ for all $\alpha,\beta\in RG$. The antipode of $\K G$ is similarly defined and has the same properties.
	
	Note that $\overline{\alpha}^\ast=\overline{\alpha^\ast}$ for all $\alpha\in RG$ and $\sigma(\beta)^\ast=\sigma(\beta^\ast)$ for all $\beta\in kG$.
\end{nothing}

\begin{nothing}\label{noth:blocks}\textit{Blocks.}
	The primitive idempotents of $Z(RG)$ are called the \textit{block idempotents} of $RG$ and the set of block idempotents is denoted $\bli(RG)$. If $b\in\bli(RG)$ then the ideal $RGb$ generated by $b$, which is itself an $R$-algebra with identity $b$, is a \textit{block algebra} of $RG$. The set of block algebras is denoted $\Bl(RG)$. Since the sets of block idempotents and block algebras of $RG$ are in natural $1$-to-$1$ correspondence, we may use the term \textit{block} to refer to either of these notions in what follows. 
	
	The maps $\overline{\cdot}:\OO G\onto kG$ and $(\cdot)^\ast:RG\to RG$ of \ref{noth:pmodsys} induce bijections
	\begin{align*}
		\bli(\OO G)	&\isoto\bli(kG)			&	\bli(RG)	&\isoto\bli(RG)\\
		b			&\mapsto\overline{b}	&	b			&\mapsto b^\ast
	\end{align*}
	Likewise the map $\sigma:kG\to kG$ restricts to a permutation of the blocks of $kG$. Via the first bijection above there is an induced permutation of the blocks of $\OO G$. By abuse of notation, in either case we write
	\begin{align*}
		\bli(RG)	&\isoto\bli(RG)\\
		b			&\mapsto\sigma(b)
	\end{align*}
	for the permutation induced by $\sigma$. Thus if $b\in\bli(\OO G)$ then $\sigma(b)$ is the unique block idempotent of $\OO G$ satisfying $\overline{\sigma(b)}=\sigma(\overline{b})$. 
	
	Blocks $a$ and $b$ of $RG$ are said to be \textit{Galois conjugate} if $\sigma^n(a)=b$ for some $n\in\N$. Note that if $a,b\in\bli(\OO G)$ then $a$ and $b$ are Galois conjugate if and only if $\overline{a}$ and $\overline{b}$ are Galois conjugate.
	
	Now if $H$ is another finite group and $(\K,\OO,k)$ is large enough for $G\times H$ then, after identifying the group algebra $R[G\times H]$ with $RG\tensor_R RH$ via $(g,h)\mapsto g\tensor h$, every block idempotent of $R[G\times H]$ is uniquely of the form $a\tensor b^\ast$ for some $a\in\bli(RG)$ and $b\in\bli(RH)$.
\end{nothing}

\begin{nothing}\label{noth:Brpairs}\textit{Brauer pairs and Brauer elements.}
	An $RG$\textit{-Brauer pair} is an ordered pair $(P,e)$ where $P$ is a $p$-subgroup of $G$ and $e\in\bli(RC_G(P))$. The group $G$ acts on the set $\BP_R(G)$ of $RG$-Brauer pairs via ${}^g(P,e)=({}^gP,{}^ge)$ for $g\in G$ and $(P,e)\in\BP_R(G)$. The stabilizer in $G$ of the Brauer pair $(P,e)$ is denoted $N_G(P,e)$. Note that $PC_G(P)\subgp N_G(P,e)\subgp N_G(P)$ for any $(P,e)\in\BP_R(G)$.
	
	The set of $RG$-Brauer pairs is partially ordered \cite[Definition 3.3]{Alperin_1979}, and the action of $G$ by conjugation respects ``inclusion'' of Brauer pairs. In other words, $\BP_R(G)$ is a $G$-poset. If $(P,e)\in\BP_R(G)$ and $Q\subgp P$ then there exists a unique block $f\in\bli(RC_G(Q))$ such that $(Q,f)\leq(P,e)$. 
	
	If $b\in\bli(RG)$ then $(P,e)$ is a $b$\textit{-Brauer pair} if $(\set{1},b)\leq(P,e)$. The set $\BP_R(G,b)$ of $b$-Brauer pairs is a subposet of $\BP_R(G)$ that is stable under the action of $G$. Every $RG$-Brauer pair is a $b$-Brauer pair for a unique block $b$ of $RG$ --- in other words, $\BP_R(G)$ is the disjoint union of the subsets $\BP_R(G,b)$, $b\in\bli(RG)$. Any two maximal $b$-Brauer pairs are $G$-conjugate, and if $(D,e)$ is a maximal $b$-Brauer pair then $D$ is a defect group of the block $B=RGb$. 
	
	One has an isomorphism of $G$-posets $\BP_\OO(G)\isoto\BP_k(G)$, $(P,e)\mapsto(P,\overline{e})$, which restricts to a $G$-poset isomorphism $\BP_\OO(G,b)\isoto\BP_k(G,\overline{b})$ for any block $b$ of $\OO G$.
	
	A \textit{Brauer element} of $RG$ is an ordered pair $(u,e)$ where $u$ is a $p$-element of $G$ and $e$ is a block idempotent of $RC_G(u)$. The set of Brauer elements of $RG$ is denoted $\BE_R(G)$. If $b$ is a block of $RG$ and $(u,e)\in\BE_R(G)$ then $(u,e)$ is a $b$\textit{-Brauer element} if $(\gp{u},e)$ is a $b$-Brauer pair. The set of $b$-Brauer elements is denoted $\BE_R(G,b)$.
\end{nothing}

\begin{nothing}\label{noth:fussys}\textit{Fusion systems.}
	Some of the results of this note require the language of \textit{fusion systems}. Rather than give the precise definition, here we will only recall two of the main situations in which these categories arise. Both types of fusion systems described below are examples of \textit{saturated} fusion systems. We refer the reader to \cite{Aschbacher_2011} for the abstract definition and the basic properties of fusion systems.
	
	Let $G$ be a finite group and let $D\in\Syl_p(G)$. Then the fusion system of $G$ over $D$ is the subcategory $\mathcal{F}_D(G)$ of the category of finite groups whose objects are the subgroups of $D$ and whose morphism sets $\Hom_{\mathcal{F}_D(G)}(P,Q)$, for $P,Q\subgp D$, consist of all group homomorphisms of the form $c_g:P\to Q$, $x\mapsto {}^gx$, where $g\in G$ is such that ${}^gP\subgp Q$.
	
	Let $G$ be a finite group, let $b\in\bli(RG)$, and let $(D,e_D)\in\BP_R(G,b)$ be a maximal $b$-Brauer pair. For each $P\subgp D$ let $e_P$ denote the unique block idempotent of $RC_G(P)$ such that $(P,e_P)\leq(D,e_D)$. The fusion system of $b$ associated to $(D,e_D)$ is the subcategory $\mathcal{F}=\mathcal{F}_{(D,e_D)}(G,b)$ of the category of finite groups whose objects are the subgroups of $D$ and with morphism sets $\Hom_{\mathcal{F}}(P,Q)$, for $P,Q\subgp D$, defined as the set of group homomorphisms $\psi:P\to Q$ for which there exists an element $g\in G$ satisfying $\psi=c_g$ and ${}^g(P,e_P)\leq(Q,e_Q)$. Note that if $b\in\bli(\OO G)$ then $\mathcal{F}_{(D,e_D)}(G,b)=\mathcal{F}_{(D,\overline{e_D})}(G,\overline{b})$.
	
	Let $\mathcal{F}$ be a fusion system over a finite $p$-group $D$ and let $P\subgp D$. Then $P$ is $\mathcal{F}$\textit{-centric} (respectively, \textit{fully} $\mathcal{F}$\textit{-centralized}) if $C_D(Q)=Z(Q)$ (resp., if $|C_D(P)|\geq|C_D(Q)|$) for all $Q\subgp D$ that are $\mathcal{F}$-isomorphic to $P$. The subgroup $P$ is \textit{normal} in $\mathcal{F}$ (written $P\nor\mathcal{F}$) if $P\nor D$ and for all $U,V\subgp D$ and all $\varphi\in\Hom_{\mathcal{F}}(U,V)$ there exists a morphism $\overline{\varphi}\in\Hom_{\mathcal{F}}(PU,PV)$ extending $\varphi$ and satisfying $\overline{\varphi}(P)=P$. There is a unique maximal normal subgroup in $\mathcal{F}$ which is denoted $O_p(\mathcal{F})$.
	
	For any subgroup $P\subgp D$ set $\Out_{\mathcal{F}}(P)=\Aut_{\mathcal{F}}(P)/\Inn(P)$. Note that if $\mathcal{F}=\mathcal{F}_D(G)$ for some finite group $G$ such that $D\in\Syl_p(G)$ then $\Out_{\mathcal{F}}(P)\iso N_G(P)/PC_G(P)$, and if $\mathcal{F}=\mathcal{F}_{(D,e_D)}(G,b)$ for some block $b$ of $RG$ then $\Out_{\mathcal{F}}(P)\iso N_G(P,e_P)/PC_G(P)$.
	
	If $\mathcal{F}'$ is another fusion system over a finite $p$-group $E$, then a group isomorphism $\phi:E\isoto D$ induces an isomorphism of fusion systems $\phi:\mathcal{F}'\isoto\mathcal{F}$ if
	\begin{equation*}
		\Hom_{\mathcal{F}}(\phi(P),\phi(Q))=\phi\circ\Hom_{\mathcal{F}'}(P,Q)\circ\phi^{-1}
	\end{equation*}
	for all $P,Q\subgp E$.
\end{nothing}


\begin{lemmadefn}\label{lem:one}
	(\cite{Kessar_2012}) The map $\BP_k(G)\to\BP_k(G)$, $(P,e)\mapsto(P,\sigma(e))$, is a $G$-poset automorphism that restricts to a $G$-poset isomorphism
	\begin{equation*}
		\BP_k(G,b)\isoto\BP_k(G,\sigma(b))
	\end{equation*}
	for any $b\in\bli(kG)$. In particular, if $b\in\bli(kG)$ and $(D,e_D)$ is a maximal $b$-Brauer pair then $(D,\sigma(e_D))$ is a maximal $\sigma(b)$-Brauer pair and
	\begin{equation*}
		\mathcal{F}_{(D,e_D)}(G,b)=\mathcal{F}_{(D,\sigma(e_D))}(G,\sigma(b)).
	\end{equation*}
	By the remarks of \ref{noth:Brpairs} we obtain a $G$-poset automorphism of $\BP_\OO(G)$, denoted abusively by $(P,e)\mapsto(P,\sigma(e))$, where if $(P,e)$ is an $\OO G$-Brauer pair then $\sigma(e)$ is the unique block idempotent of $\OO C_G(P)$ such that $\overline{\sigma(e)}=\sigma(\overline{e})$. Moreover this map restricts to a $G$-poset isomorphism $\BP_\OO(G,b)\isoto\BP_\OO(G,\sigma(b))$ for any $b\in\bli(\OO G)$.
\end{lemmadefn}


\begin{nothing}\label{noth:selfcent}\textit{Self-centralizing Brauer pairs and K\"{u}lshammer-Puig classes.}
	Let $b\in\bli(RG)$ and let $(P,e)\in\BP_R(G,b)$. Then $e$ is a block of $R[PC_G(P)]$, $(P,e)\in\BP_R(PC_G(P),e)$ and $(Z(P),e)\in\BP_R(C_G(P),e)$. The Brauer pair $(P,e)$ is called \textit{self-centralizing} if $(P,e)$ is maximal in $\BP_R(PC_G(P),e)$ or, equivalently, if $(Z(P),e)$ is maximal in $\BP_R(C_G(P),e)$. Additionally, if $(D,e_D)$ is a maximal $b$-Brauer pair such that $(P,e)\leq(D,e_D)$ then $(P,e)$ is self-centralizing if and only if $P$ is an $\mathcal{F}$-centric subgroup of $D$, where $\mathcal{F}=\mathcal{F}_{(D,e_D)}(G,b)$ \cite[Theorem IV.3.20]{Aschbacher_2011}.
	
	Let $(P,e)$ be a self-centralizing $RG$-Brauer pair. Set $I=N_G(P,e)$ and $\overline{I}=I/PC_G(P)$. By \cite[Lemma 5.8.12]{Nagao_1989} the block algebra $k[PC_G(P)]\overline{e}$ possesses a unique (absolutely) irreducible module $V$, which is necessarily $I$-stable. Therefore $V$ determines (cf. \cite[Theorem 3.5.7]{Nagao_1989}) a cohomology class $\kappa_{(P,e)}\in H^2(\overline{I},k^\times)$ called the \textit{K\"{u}lshammer-Puig class} at $(P,e)$.
	
	An alternate definition of the K\"{u}lshammer-Puig class is provided in \cite[Definition IV.5.31]{Aschbacher_2011}. We note that this definition is essentially equivalent to the one we have provided, except that the cohomology class which is produced in \cite{Aschbacher_2011} is the inverse of the class defined above.
\end{nothing}

\begin{lemma}\label{lem:two}
	Let $(P,e)\in\BP_R(G)$. Then $(P,e)$ is self-centralizing if and only if $(P,\sigma(e))$ is self-centralizing.
\end{lemma}

\begin{proof}
	By Lemma \ref{lem:one} the posets $\BP_R(PC_G(P),e)$ and $\BP_R(PC_G(P),\sigma(e))$ are isomorphic.
\end{proof}


\begin{lemma}\label{lem:forexample}
	(\cite[Proposition IV.5.35]{Aschbacher_2011}) Let $\mathcal{F}$ be a saturated fusion system over a finite $p$-group $D$ and suppose $P$ is an $\mathcal{F}$-centric subgroup of $D$ such that $P\nor\mathcal{F}$. Then for any $\kappa\in H^2(\Out_{\mathcal{F}}(P),k^\times)$ there exists a finite group $G$ containing $D$ and a block $b$ of $\OO G$ such that $(Q,b)$ is a $b$-Brauer pair for any $p$-subgroup $Q\subgp G$, $(D,b)$ is a maximal $b$-Brauer pair, $\mathcal{F}=\mathcal{F}_{(D,b)}(G,b)$ and the K\"{u}lshammer-Puig class at $(P,b)$ is $\kappa$.
\end{lemma}

\begin{nothing}\label{noth:clfuns}\textit{Class functions.}
	We write $CF(G;\K)$ for the $\K$-algebra of $\K$-valued class functions on $G$. The space of class functions is equipped with the usual inner product for which the set $\Irr_\K(G)$ of characters of irreducible $\K G$-modules is an orthonormal basis. We identify each class function on $G$ with its linear extension to $\K G$. Write $R_\K(G)$ for the $\Z$-span of $\Irr_\K(G)$ within $CF(G;\K)$. Let $CF(G;\OO)$ denote the $\OO$-subalgebra of $CF(G;\K)$ formed by the $\OO$-valued class functions and let $CF_{p'}(G;\K)$ denote the subspace of $CF(G;\K)$ consisting of class functions $\chi$ that satisfy $\chi(g)=0$ for all $g\in G-G_{p'}$, where $G_{p'}$ denotes the set of $p'$-elements of $G$.
	
	If $e$ is a central idempotent of $\OO G$ let $CF(G,e;\K)$ denote the subspace of class functions $\chi\in CF(G;\K)$ that satisfy $\chi(ge)=\chi(g)$ for all $g\in G$. The set $\Irr_\K(G,e)$ of characters of irreducible $\K Ge$-modules is a basis of $CF(G,e;\K)$. Let $R_\K(G,e)$ denote the subgroup of $R_\K(G)$ spanned by $\Irr_\K(G,e)$, let
	\begin{equation*}
		CF(G,e;\OO)=CF(G;\OO)\cap CF(G,e;\K)
	\end{equation*}
	and let
	\begin{equation*}
		CF_{p'}(G,e;\K)=CF_{p'}(G;\K)\cap CF(G,e;\K).
	\end{equation*}

	If $H\subgp G$, $g\in G$ and $\chi\in CF(H;\K)$ then ${}^g\chi\in CF({}^gH;\K)$ is the class function defined by ${}^g\chi({}^gh)=\chi(h)$ for all $h\in H$.
\end{nothing}

\begin{nothing}\label{noth:gendec}\textit{The generalized decomposition map.}
	If $(u,e)\in\BE_\OO(G)$ (see \ref{noth:Brpairs} above) then the generalized decomposition map associated to $(u,e)$ is the $\K$-linear map
	\begin{equation*}
		d_G^{u,e}:CF(G;\K)\to CF_{p'}(C_G(u),e;\K)
	\end{equation*}
	defined by
	\begin{equation*}
		d_G^{u,e}(\chi)(s)=\begin{cases}
			\chi(use)	&\text{if }s\in C_G(u)_{p'}\\
			0			&\text{if }s\notin C_G(u)_{p'}
		\end{cases}
	\end{equation*}
	for all $\chi\in CF(G;\K)$ and $s\in C_G(u)$.
\end{nothing}

\begin{nothing}\label{noth:isot}\textit{Isotypies.}
	Let $G$ and $H$ be finite groups and assume that $(\K,\OO,k)$ is large enough for $G$ and $H$. Let $A=\OO Ga\in\Bl(\OO G)$ and $B=\OO Hb\in\Bl(\OO H)$. Let $(D,e_D)\in\BP_\OO(G,a)$ and $(E,f_E)\in\BP_\OO(H,b)$ be maximal Brauer pairs, and for each $P\subgp D$ (respectively, $Q\subgp E$) let $e_P$ (resp. $f_Q$) denote the unique block idempotent of $\OO C_G(P)$ (resp. $\OO C_H(Q)$) such that $(P,e_P)\leq(D,e_D)$ (resp. $(Q,f_Q)\leq(E,f_E)$). Assume $\phi:E\isoto D$ is a group isomorphism that induces an isomorphism of fusion systems $\phi:\mathcal{F}_{(E,f_E)}(H,b)\isoto\mathcal{F}_{(D,e_D)}(G,a)$. 
	
	In this situation, an isotypy between $A$ and $B$ (relative to $(D,e_D)$, $(E,f_E)$ and $\phi:E\isoto D$) is a family of isometries
	\begin{equation*}
		I^Q:R_\K(C_H(Q),f_Q)\isoto R_\K(C_G(\phi(Q)),e_{\phi(Q)}),\qquad Q\subgp E
	\end{equation*}
	satisfying
	\begin{itemize}
		\item (Equivariance) If $Q\subgp E$ and $(g,h)\in G\times H$ is such that
		\begin{equation*}
			{}^{(g,h)}(\Delta(\phi(Q),\phi,Q),e_{\phi(Q)}\tensor f_Q^\ast)\leq(\Delta(D,\phi,E),e_D\tensor f_E^\ast)
		\end{equation*}
		in $\BP_\OO(G\times H,a\tensor b^\ast)$ then ${}^gI^Q(\psi)=I^{{}^hQ}({}^h\psi)$ for any $\psi\in R_\K(C_H(Q),f_Q)$.
		\item (Compatibility) For every $Q\subgp E$ and $v\in C_E(Q)$ the diagram
		\[\begin{tikzcd}
			{CF(C_H(Q),f_Q;\K)} &&& {CF(C_G(P),e_P;\K)} \\
			{CF(C_H(Q\gp{v}),f_{Q\gp{v}};\K)} &&& {CF(C_G(P\gp{u}),e_{P\gp{u}};\K)}
			\arrow["{I^Q}", from=1-1, to=1-4]
			\arrow["{d_{C_H(Q)}^{v,f_{Q\gp{v}}}}"', from=1-1, to=2-1]
			\arrow["{d_{C_G(P)}^{u,e_{P\gp{u}}}}", from=1-4, to=2-4]
			\arrow["{I^{Q\gp{v}}}"', from=2-1, to=2-4]
		\end{tikzcd}\]
		commutes, where $P=\phi(Q)$, $u=\phi(v)$, and where $I^Q$ and $I^{Q\gp{v}}$ denote (abusively) the $\K$-linear extensions of the isometries $I^Q$ and $I^{Q\gp{v}}$, respectively.
	\end{itemize}

	It can be shown (cf. \cite{Broue_1995}) that if $I^Q$, $Q\subgp E$, is an isotypy between $A$ and $B$ then each isometry $I^Q$ is \textit{perfect}, i.e., defines a $\K$-isomorphism
	\begin{equation*}
		I^Q:CF_{p'}(C_H(Q),f_Q;\K)\isoto CF_{p'}(C_G(\phi(Q)),e_{\phi(Q)};\K)
	\end{equation*}
	and an $\OO$-isomorphism
	\begin{equation*}
		I^Q:CF(C_H(Q),f_Q;\OO)\isoto CF(C_G(\phi(Q)),e_{\phi(Q)};\OO).
	\end{equation*}
\end{nothing}

The next theorem is the main result of \cite{Kessar_2012}. Since the definition of an isotypy between blocks used here is stronger than that of \cite{Kessar_2012} --- in the sense that we require an isometry for each subgroup of the defect group, rather than each \textit{cyclic} subgroup --- in the subsequent proof we provide some information on how Kessar's result is easily extended to our setting.

\begin{theorem}\label{thm:one}
	(\cite[Theorem 1.1]{Kessar_2012}) Let $b\in\bli(\OO G)$, let $(D,e_D)\in\BP_\OO(G,b)$ be a maximal $b$-Brauer pair and for each $P\subgp D$ let $e_P$ denote the unique block idempotent of $\OO C_G(P)$ such that $(P,e_P)\leq(D,e_D)$. By Lemma \ref{lem:one}, $(D,\sigma(e_D))$ is a maximal $\sigma(b)$-Brauer pair and if $P\subgp D$ then $(P,\sigma(e_P))\leq(D,\sigma(e_D))$. For each $P\subgp D$ the map
	\begin{equation*}
		I^P:R_\K(C_G(P),e_P)\to R_\K(C_G(P),\sigma(e_P))
	\end{equation*}
	defined by $I^P(\chi)(g)=\chi(g_pg_{p'}^p)$ for $\chi\in R_\K(C_G(P),e_P)$ and $g\in C_G(P)$, is a well-defined (perfect) isometry. Furthermore the collection of isometries $I^P$, $P\subgp D$, form an isotypy between $\OO G\sigma(b)$ and $\OO Gb$ relative to $(D,\sigma(e_D))$, $(D,e_D)$, and $\id_D:D\isoto D$.
\end{theorem}


\begin{proof}
	Since
	\begin{equation*}
		\mathcal{F}_{(D,e_D)}(G,b)=\mathcal{F}_{(D,\overline{e_D})}(G,\overline{b})=\mathcal{F}_{(D,\sigma(\overline{e_D}))}(G,\sigma(\overline{b}))=\mathcal{F}_{(D,\sigma(e_D))}(G,\sigma(b))
	\end{equation*}
	by Lemma \ref{lem:one}, the identity of $D$ does in fact define an isomorphism of fusion systems and so the hypotheses for the existence of an isotypy between $\OO G\sigma(b)$ and $\OO Gb$ are satisfied. Let $\mathcal{F}=\mathcal{F}_{(D,e_D)}(G,b)$.
	
	The maps $I^P$, $P\subgp D$, are well-defined isometries by \cite[Lemma 3.3]{Kessar_2012}, and verifying that these isometries satisfy the equivariance condition is straightforward. Now since the isometries $I^P$, $P\subgp D$, are equivariant, to verify the compatiblity condition it is enough to show that
	\begin{equation}\label{eqn:1}\tag{$\ast$}
		d_{C_G(P)}^{u,\sigma(e_{P\gp{u}})}(I^P(\psi))=I^{P\gp{u}}(d_{C_G(P)}^{u,e_{P\gp{u}}}(\psi))
	\end{equation}
	for each fully $\mathcal{F}$-centralized subgroup $P\subgp D$, each $u\in C_D(P)$ and each $\psi\in CF(C_G(P),e_P;\K)$. If $P\subgp D$ is fully $\mathcal{F}$-centralized then by \cite[Theorem IV.3.19(b)]{Aschbacher_2011} $(C_D(P),e_{PC_D(P)})$ is a maximal $\OO C_G(P)e_P$-Brauer pair and if $Q\subgp C_D(P)$ then $(Q,e_{PQ})\leq(C_D(P),e_{PC_D(P)})$. By Lemma \ref{lem:one} $(C_D(P),\sigma(e_{PC_D(P)}))$ is a maximal $\OO C_G(P)\sigma(e_P)$-Brauer pair and
	\begin{equation*}
		\mathcal{F}_{(C_D(P),e_{PC_D(P)})}(C_G(P),e_P)=\mathcal{F}_{(C_D(P),\sigma(e_{PC_D(P)}))}(C_G(P),\sigma(e_P)).
	\end{equation*}
	The equality (\ref{eqn:1}) now follows directly from the proof of \cite[Theorem 1.1]{Kessar_2012}, with $G$ replaced by $C_G(P)$ and $b$ replaced by $e_P$.
\end{proof}

\begin{nothing}\label{noth:ppermequiv}$p$\textit{-permutation equivalences.}
	Recall that an $RG$-module $M$ is called a \textit{trivial source} or $p$\textit{-permutation module} if $\Res_P^GM$ is a permutation $RP$-module for all $p$-subgroups $P\subgp G$. The class of trivial source modules is closed under $\dsum$, $\tensor_R$, taking direct summands, among other operations. The Grothendieck ring of the category of trivial source $RG$-modules (taken with respect to split short exact sequences) is denoted $T_R(G)$. As an abelian group, $T_R(G)$ is free of finite rank with standard basis given by the isomorphism classes $[M]$ of indecomposable trivial source $RG$-modules $M$. If $e$ is a central idempotent of $RG$ then $T_R(G,e)$ denotes the subgroup of $T_R(G)$ spanned by the isomorphism classes $[M]$ of indecomposable trivial source $RGe$-modules $M$. The functor $k\tensor_{\OO}-$ induces a ring isomorphism $T_\OO(G)\isoto T_k(G)$ that maps standard basis elements to standard basis elements and restricts to a group isomorphism $T_\OO(G,e)\isoto T_k(G,\overline{e})$.
	
	If $H$ is another finite group and $e$ and $f$ are nonzero central idempotents of $\OO G$ and $\OO H$, respectively, then Boltje and Perepelitsky \cite{Boltje_2020} define a $p$\textit{-permutation equivalence} between $\OO Ge$ and $\OO Hf$ to be an element $\gamma\in T_\OO(G\times H,e\tensor f^\ast)$ satisfying
	\begin{itemize}
		\item[(1)] $\gamma$ is a $\Z$-linear combination of isomorphism classes of indecomposable trivial source $\OO[G\times H](e\tensor f^\ast)$-modules that have twisted diagonal vertices,
		\item[(2)] $\gamma\tensor_{\OO H}\gamma^\circ=[\OO Ge]\in T_\OO(G\times G,e\tensor e^\ast)$, and
		\item[(3)] $\gamma^\circ\tensor_{\OO G}\gamma=[\OO Hf]\in T_\OO(H\times H,f\tensor f^\ast)$,
	\end{itemize} 
	where $\gamma^\circ$ denotes the $\OO$-dual of $\gamma$.
	
	Let $A=\OO Ga\in\Bl(\OO G)$, $B=\OO Hb\in\Bl(\OO H)$ and let $\gamma$ be a $p$-permutation equivalence between $A$ and $B$. Boltje and Perepelitsky associate to $\gamma$ a subset $\BP_\OO(\gamma)\subseteq \BP_\OO(G\times H,a\tensor b^\ast)$, whose elements they call $\gamma$\textit{-Brauer pairs}. They show \cite[Theorem 10.11]{Boltje_2020} that the set of $\gamma$-Brauer pairs is stable under $G\times H$-conjugation, that $\BP_\OO(\gamma)$ is an ideal of the poset $\BP_\OO(G\times H,a\tensor b^\ast)$, and that any two maximal $\gamma$-Brauer pairs are $G\times H$-conjugate. Moreover any $\gamma$-Brauer pair is of the form $(\Delta(P,\psi,Q),e\tensor f^\ast)$ where $\Delta(P,\psi,Q)$ is a twisted diagonal subgroup of $G\times H$, $e\in\bli(\OO C_G(P))$ and $f\in\bli(\OO C_H(Q))$.
\end{nothing}

\begin{nothing}\label{noth:isofrompperm}\textit{Isotypies from }$p$\textit{-permutation equivalences.}
	Let $G$ and $H$ be finite groups, $A=\OO Ga\in\Bl(\OO G)$, $B=\OO Hb\in\Bl(\OO H)$, and let $\gamma$ be a $p$-permutation equivalence between $A$ and $B$. Choose a maximal $\gamma$-Brauer pair $(\Delta(D,\phi,E),e_D\tensor f_E^\ast)$. Then $(D,e_D)$ is a maximal $a$-Brauer pair, $(E,f_E)$ is a maximal $b$-Brauer pair, and $\phi:E\isoto D$ induces an isomorphism of fusion systems $\phi:\mathcal{F}_{(E,f_E)}(H,b)\isoto\mathcal{F}_{(D,e_D)}(G,a)$ \cite[Theorems 10.11(c), 11.2]{Boltje_2020}. For each $P\subgp D$ (respectively, $Q\subgp E$) let $e_P$ (resp. $f_Q$) denote the unique block idempotent of $\OO C_G(P)$ (resp. $\OO C_H(Q)$) such that $(P,e_P)\leq (D,e_D)$ (resp. $(Q,f_Q)\leq(E,f_E)$). Then for each $Q\subgp E$,
	\begin{equation*}
		(\Delta(\phi(Q),\phi,Q),e_{\phi(Q)}\tensor f_Q^\ast)\leq(\Delta(D,\phi,E),e_D\tensor f_E^\ast)
	\end{equation*}
	is a containment of $a\tensor b^\ast$-Brauer pairs.
	
	The $p$-permutation equivalence $\gamma$ induces an isotypy between $A$ and $B$ (relative to $(D,e_D)$, $(E,f_E)$ and $\phi:E\isoto D$) via the following construction: for each $Q\subgp E$, applying the \textit{Brauer construction} (cf. \cite[Chapter 5.4]{Linckelmann_2018}) at $\Delta(\phi(Q),\phi,Q)$ and lifting to $\OO$ yields an element
	\begin{equation*}
		\gamma(\Delta(\phi(Q),\phi,Q))\in T_\OO(N_{G\times H}(\Delta(\phi(Q),\phi,Q))).
	\end{equation*}
	Restricting this element to $C_G(\phi(Q))\times C_H(Q)$ and taking the $e_{\phi(Q)}\tensor f_Q^\ast$-component of the restriction produces an element
	\begin{equation*}
		\gamma(\Delta(\phi(Q),\phi,Q),e_{\phi(Q)}\tensor f_Q^\ast)\in T_\OO(C_G(\phi(Q))\times C_H(Q),e_{\phi(Q)}\tensor f_Q^\ast).
	\end{equation*}
	Let $\mu_Q$ denote the character of $\gamma(\Delta(\phi(Q),\phi,Q),e_{\phi(Q)}\tensor f_Q^\ast)$ and let $I^Q:R_\K(C_H(Q),f_Q)\to R_\K(C_G(\phi(Q)),e_{\phi(Q)})$ be the map $I^Q=\mu_Q\tensor_{\K C_H(Q)}-$. By \cite[Theorem 15.4]{Boltje_2020}, the collection of maps $I^Q$, $Q\subgp E$, form an isotypy between $A$ and $B$.
\end{nothing}

\begin{definition}\label{defn:lift}
	Let $G$ and $H$ be finite groups, $A=\OO Ga\in\Bl(\OO G)$ and $B=\OO Hb\in\Bl(\OO H)$. Let $I^Q$, $Q\subgp E$, be an isotypy between $A$ and $B$ defined relative to a maximal $a$-Brauer pair $(D,e_D)$, a maximal $b$-Brauer pair $(E,f_E)$, and an isomorphism $\phi:E\isoto D$ that induces an isomorphism $\phi:\mathcal{F}_{(E,f_E)}(H,b)\isoto\mathcal{F}_{(D,e_D)}(G,a)$. We say the isotypy $I^Q$, $Q\subgp E$, \textit{lifts} to a $p$-permutation equivalence between $A$ and $B$ if there exists a $p$-permutation equivalence $\gamma\in T_\OO(G\times H,a\tensor b^\ast)$ that possesses $(\Delta(D,\phi,E),e_D\tensor f_E^\ast)$ as a maximal $\gamma$-Brauer pair and such that the isotypy $I^Q$, $Q\subgp E$, is obtained from $\gamma$ via the construction outlined in \ref{noth:isofrompperm} above.
\end{definition}

\begin{nothing}\label{noth:presKP}\textit{Preservation of K\"{u}lshammer-Puig classes.}
	Let $A\in\Bl(\OO G)$, $B\in\Bl(\OO H)$ and let $\gamma$ be a $p$-permutation equivalence between $A$ and $B$. Suppose  $(\Delta(P,\psi,Q),e\tensor f^\ast)\in\BP_\OO(\gamma)$ is such that $(Q,f)$ is a self-centralizing $\OO H$-Brauer pair. Then $(P,e)$ is a self-centralizing $\OO G$-Brauer pair by \cite[Theorem 11.2]{Boltje_2020}. Let $I=N_G(P,e)$, $\overline{I}=I/PC_G(P)$, $J=N_H(Q,f)$ and $\overline{J}=J/QC_H(Q)$. One has a group isomorphism $\eta:J/C_H(Q)\isoto I/C_G(P)$ which maps $hC_H(Q)$ to $gC_G(P)$ if and only if $(g,h)\in N_{G\times H}(\Delta(P,\psi,Q),e\tensor f^\ast)$ \cite[Proposition 11.1]{Boltje_2020}. The map $\eta$ descends to a group isomorphism $\overline{\eta}:\overline{J}\isoto\overline{I}$, which in turn induces a group isomorphism $\overline{\eta}^\ast:H^2(\overline{I},k^\times)\isoto H^2(\overline{J},k^\times)$. If $\kappa_{(P,e)}\in H^2(\overline{I},k^\times)$ and $\kappa_{(Q,f)}\in H^2(\overline{J},k^\times)$ denote the K\"{u}lshammer-Puig classes at $(P,e)$ and $(Q,f)$, respectively, then $\kappa_{(Q,f)}=\overline{\eta}^\ast(\kappa_{(P,e)})$ \cite[Theorem 13.4]{Boltje_2020}.
\end{nothing}

\section{Results}

Keep all of the notation of the previous section --- in particular $G$ is a finite group, $(\K,\OO,k)$ is a $p$-modular system large enough for $G$ such that $k$ is algebraically closed, $R$ denotes either of the rings $\OO$ or $k$, and $\sigma:k\isoto k$ is the Frobenius automorphism of $k$. Recall that $\sigma$ induces an automorphism of the $n$th cohomology group $H^n(G,k^\times)$ for any $n\geq 0$, defined by $\kappa\mapsto\kappa^p$ for each $\kappa\in H^n(G,k^\times)$. We denote this automorphism of $H^n(G,k^\times)$ abusively by $\sigma$ in the sequel.

\begin{lemma}\label{lem:three}
	Let $(P,e)$ be a self-centralizing $RG$-Brauer pair. Then by Lemma \ref{lem:two}, $(P,\sigma(e))$ is self-centralizing. Let $I=N_G(P,e)=N_G(P,\sigma(e))$ and set $\overline{I}=I/PC_G(P)$. Let $\kappa_{(P,e)},\kappa_{(P,\sigma(e))}\in H^2(\overline{I},k^\times)$ denote the K\"{u}lshammer-Puig classes at $(P,e)$ and $(P,\sigma(e))$, respectively. Then $\kappa_{(P,\sigma(e))}=\sigma(\kappa_{(P,e)})$, i.e., $\kappa_{(P,\sigma(e))}=\kappa_{(P,e)}^p$.
\end{lemma}

\begin{proof}
	Let $Y:PC_G(P)\to GL_n(k)$ be a representation that affords the unique irreducible $k[PC_G(P)]\overline{e}$-module and let $X:I\to GL_n(k)$ be a projective representation extending $Y$ that satisfies $X(hx)=X(h)X(x)$ and $X(xh)=X(x)X(h)$ for all $h\in PC_G(P)$, $x\in I$. The factor set $\alpha$ associated to $X$ descends to a factor set $\overline{\alpha}\in Z^2(\overline{I},k^\times)$ whose image in $H^2(\overline{I},k^\times)$ is $\kappa_{(P,e)}$.
	
	Let $F:GL_n(k)\to GL_n(k)$ denote the standard Frobenius map $(a_{ij})\mapsto(a_{ij}^p)$. Then $F\circ Y:PC_G(P)\to GL_n(k)$ affords an irreducible $k[PC_G(P)]\sigma(\overline{e})$-module, $F\circ X$ is a projective representation of $I$ extending $F\circ Y$, and we have $(FX)(hx)=(FX)(h)(FX)(x)$ and $(FX)(xh)=(FX)(x)(FX)(h)$ for all $h\in PC_G(P)$, $x\in I$.
	
	Now if $x,y\in I$ then
	\begin{align*}
		(FX)(x)(FX)(y)	&=F(X(x)X(y))=F(\alpha(x,y)X(xy))\\
						&=\alpha(x,y)^p(FX)(xy),
	\end{align*}
	so the factor set associated to $F\circ X$ is $\alpha^p$. This factor set descends to $\overline{\alpha^p}\in Z^2(\overline{I},k^\times)$, and the image of $\overline{\alpha^p}$ in $H^2(\overline{I},k^\times)$ is $\kappa_{(P,\sigma(e))}$. Since $\overline{\alpha^p}=\overline{\alpha}^p$ we have $\kappa_{(P,\sigma(e))}=\kappa_{(P,e)}^p$ and the proof is complete.
\end{proof}


\begin{theorem}\label{thm:two}
	Let $b\in\bli(\OO G)$ and let $(D,e_D)\in\BP_\OO(G,b)$ be a maximal $b$-Brauer pair. For each $P\subgp D$ let $e_P$ denote the unique block idempotent of $\OO C_G(P)$ such that $(P,e_P)\leq(D,e_D)$ and let $I_P=N_G(P,e_P)$. Suppose there exists a subgroup $P\subgp D$ such that $(P,e_P)$ is self-centralizing and the K\"{u}lshammer-Puig class $\kappa_{(P,e_P)}$ is not fixed by the Frobenius automorphism of $H^2(I_P/PC_G(P),k^\times)$, i.e., $\kappa_{(P,e_P)}\neq\kappa_{(P,e_P)}^p$. Then the isotypy of Theorem \ref{thm:one} does not lift to a $p$-permutation equivalence between $\OO G\sigma(b)$ and $\OO Gb$.
\end{theorem}

\begin{proof}
	Recall that the isotypy of Theorem \ref{thm:one} is defined relative to the maximal $\sigma(b)$-Brauer pair $(D,\sigma(e_D))$, the maximal $b$-Brauer pair $(D,e_D)$, and the identity automorphism of $D$. Suppose this isotypy lifts to a $p$-permutation equivalence $\gamma$ between $\OO G\sigma(b)$ and $\OO Gb$. Then by Definition \ref{defn:lift} $(\Delta(D),\sigma(e_D)\tensor e_D^\ast)$ is a maximal $\gamma$-Brauer pair. 
	
	Now $(\Delta(P),\sigma(e_P)\tensor e_P^\ast)\leq(\Delta(D),\sigma(e_D)\tensor e_D^\ast)$, so $(\Delta(P),\sigma(e_P)\tensor e_P^\ast)$ is a $\gamma$-Brauer pair. Recall from the remarks of \ref{noth:presKP} that in this case there is a group automorphism $\eta:I_P/C_G(P)\isoto I_P/C_G(P)$ that induces an automorphism $\overline{\eta}^\ast$ of $H^2(I_P/PC_G(P),k^\times)$ mapping $\kappa_{(P,\sigma(e_P))}$ to $\kappa_{(P,e_P)}$. If $g,h\in I_P$ then $\eta(hC_G(P))=gC_G(P)$ if and only if $(g,h)\in N_{G\times G}(\Delta(P),\sigma(e_P)\tensor e_P^\ast)$. But if $(g,h)$ normalizes the diagonal subgroup $\Delta(P)$ then $gh^{-1}\in C_G(P)$, so $\eta$ is the identity automorphism of $I_P/C_G(P)$. It follows that $\overline{\eta}^\ast$ is the identity of $H^2(I_P/PC_G(P),k^\times)$ and $\kappa_{(P,\sigma(e_P))}=\kappa_{(P,e_P)}$. On the other hand, $\kappa_{(P,\sigma(e_P))}=\kappa_{(P,e_P)}^p$ by Lemma \ref{lem:three} and hence $\kappa_{(P,e_P)}=\kappa_{(P,e_P)}^p$, a contradiction.
\end{proof}

Blocks as in Theorem \ref{thm:two} exist: see Example \ref{ex:one} below. In fact, the blocks constructed in Example \ref{ex:one} satisfy a stronger statement which is provided in the following theorem.

\begin{theorem}\label{thm:three}
	Let $b\in\bli(\OO G)$, let $(D,e_D)\in\BP_\OO(G,b)$ be a maximal $b$-Brauer pair, and let $\mathcal{F}=\mathcal{F}_{(D,e_D)}(G,b)$. For each $P\subgp D$ let $e_P$ denote the unique block idempotent of $\OO C_G(P)$ such that $(P,e_P)\leq(D,e_D)$ and set $I_P=N_G(P,e_P)$. Suppose there exists a subgroup $P\subgp D$ such that
	\begin{itemize}
		\item[(1)] $(P,e_P)$ is self-centralizing and $\kappa_{(P,e_P)}\neq\kappa_{(P,e_P)}^p$, where $\kappa_{(P,e_P)}$ denotes the K\"{u}lshammer-Puig class at $(P,e_P)$,
		\item[(2)] If $\phi:D\isoto D$ is an automorphism of $D$ that induces an automorphism of fusion systems $\phi:\mathcal{F}\isoto\mathcal{F}$ then $\phi|_P=c_g$ for some $g\in I_P$.
	\end{itemize}
	Then $\OO G\sigma(b)$ and $\OO Gb$ are not $p$-permutation equivalent.
\end{theorem}

\begin{proof}
	Suppose there exists a $p$-permutation equivalence $\gamma$ between $\OO G\sigma(b)$ and $\OO Gb$. By Lemma \ref{lem:one}, $(D,\sigma(e_D))$ is a maximal $\sigma(b)$-Brauer pair and $\mathcal{F}=\mathcal{F}_{(D,\sigma(e_D))}(G,\sigma(b))$. Since the maximal $\OO G$-Brauer pairs belonging to a fixed block are all $G$-conjugate, there exists a maximal $\gamma$-Brauer pair of the form $(\Delta(D,\phi,D),\sigma(e_D)\tensor e_D^\ast)$. By \cite[Theorem 11.2]{Boltje_2020}, the automorphism $\phi$ of $D$ induces an automorphism of the fusion system $\mathcal{F}$. By assumption (2) there exists an element $g\in I_P$ such that $\phi|_P=c_g$. In particular, $\phi(P)=P$ and $(\Delta(P,\phi,P),\sigma(e_P)\tensor e_P^\ast)$ is a $\gamma$-Brauer pair since it is contained in $(\Delta(D,\phi,D),\sigma(e_D)\tensor e_D^\ast)$. Now the set of $\gamma$-Brauer pairs is stable under conjugation, so
	\begin{equation*}
		{}^{(1,g)}(\Delta(P,\phi,P),\sigma(e_P)\tensor e_P^\ast)=(\Delta(P),\sigma(e_P)\tensor e_P^\ast)
	\end{equation*}
	is a $\gamma$-Brauer pair. Arguing as in the proof of Theorem \ref{thm:two}, the remarks of \ref{noth:presKP} imply that $\kappa_{(P,e_P)}=\kappa_{(P,\sigma(e_P))}$. By Lemma \ref{lem:three} we have $\kappa_{(P,\sigma(e_P))}=\kappa_{(P,e_P)}^p$ and therefore $\kappa_{(P,e_P)}=\kappa_{(P,e_P)}^p$, contradicting assumption (1). We conclude that $\OO G\sigma(b)$ and $\OO Gb$ are not $p$-permutation equivalent.
\end{proof}

\begin{example}\label{ex:one}
	We demonstrate a method for constructing blocks that satisfy the assumptions of Theorem \ref{thm:three}, which are therefore not $p$-permutation equivalent to their Galois conjugate. The blocks we will describe also satisfy the assumptions of Theorem \ref{thm:two}, and thus provide examples of blocks for which the isotypy of Theorem \ref{thm:one} does not lift to a $p$-permutation equivalence.
	
	Let $P$ be a $p$-group and let $A\subgp\Aut(P)$ be such that
	\begin{itemize}
		\item[(1)] $\Inn(P)\cap A=\set{1}$,
		\item[(2)] $O_p(A)=\set{1}$,
		\item[(3)] The image of $A$ in $\Out(P)$ is self-normalizing, and
		\item[(4)] There exists a cohomology class $\kappa\in H^2(A,k^\times)$ such that $\kappa\neq\kappa^p$.
	\end{itemize}
	For example, if $p=2$ and $P=E_{2^4}$ is the elementary abelian group of order $2^4$ then it is well-known that $\Aut(P)\iso GL_4(2)\iso A_8$, the alternating group of degree 8. Let $A$ be a subgroup of $\Aut(P)$ isomorphic to $A_7$. Properties (1), (2) and (3) above are easily verified for this choice of $P$ and $A$. Note that, since $p=2$, a cohomology class $\kappa\in H^2(A,k^\times)$ satisfies $\kappa\neq\kappa^p$ if and only if $\kappa$ is nontrivial. Now $H^2(A,k^\times)$ maps surjectively onto $\Hom(H_2(A,\Z),k^\times)$ and $H_2(A_7,\Z)\iso\Z/6\Z$ (cf. \cite[Exercises 6.1.5, Example 6.9.10]{Weibel_1994}). It follows that $H^2(A,k^\times)$ is nontrivial and property (4) holds. 
	
	
	Returning to the general case, let $P$ be a $p$-group and let $A$ be a group of automorphisms of $P$ satisfying properties (1) through (4) above. Fix a cohomology class $\kappa\in H^2(A,k^\times)$ such that $\kappa\neq\kappa^p$. Set $L=P\rtimes A$, let $S\in\Syl_p(A)$, $D=P\rtimes S\in\Syl_p(L)$, and set $\mathcal{F}=\mathcal{F}_D(L)$. Then $C_L(P)=Z(P)=C_D(P)$ by (1). Since $P$ is not $\mathcal{F}$-isomorphic to any other subgroup of $D$ it follows that $P$ is $\mathcal{F}$-centric. Now $P\nor L$ so it is clear that $P\nor\mathcal{F}$. Suppose $P\subgp Q\subgp D$ with $Q\nor\mathcal{F}$. Then any $\varphi\in\Aut_{\mathcal{F}}(P)$ extends to a morphism $\overline{\varphi}\in\Aut_{\mathcal{F}}(Q)$. This means that if $g\in L$ then there exists $h\in N_L(Q)$ such that ${}^gx={}^hx$ for all $x\in P$, i.e., $g^{-1}h\in Z(P)$. Therefore $N_L(Q)\cap gZ(P)$ is nonempty, and hence $gZ(P)\subseteq N_L(Q)$, for any $g\in L$, which implies $Q\nor L$. But then $Q/P$ is a normal $p$-subgroup of $L/P\iso A$, so property (2) forces $Q=P$. We have shown that $O_p(\mathcal{F})=P$.
	
	Let $\phi:D\isoto D$ be an automorphism that induces an automorphism of fusion systems $\phi:\mathcal{F}\isoto\mathcal{F}$. Then $\phi(P)=P$ since $P=O_p(\mathcal{F})$. Since $\phi$ defines an automorphism of $\mathcal{F}$ it follows that $\phi|_P$ normalizes $\Aut_{\mathcal{F}}(P)=\Inn(P)A$. Property (3) then implies $\phi|_P\in\Aut_{\mathcal{F}}(P)$.
	
	Now $\Out_{\mathcal{F}}(P)\iso A$, so $\kappa$ may be regarded as an element of $H^2(\Out_{\mathcal{F}}(P),k^\times)$. Then by Lemma \ref{lem:forexample} there exists a finite group $G$ containing $D$ and a block $b$ of $\OO G$ such that $(Q,b)$ is a $b$-Brauer pair for any $p$-subgroup $Q\subgp G$, $(D,b)$ is a maximal $b$-Brauer pair, $\mathcal{F}=\mathcal{F}_{(D,b)}(G,b)$ and the K\"{u}lshammer-Puig class at $(P,b)$ is $\kappa$ (we remark that the group $G$ can be taken to be a certain central extension of $L$ by a cyclic group of order $|L|_{p'}$, cf. \cite[Proposition IV.5.35]{Aschbacher_2011}). Since $P$ is $\mathcal{F}$-centric the Brauer pair $(P,b)$ is self-centralizing and by our choice of $\kappa$ we have $\kappa_{(P,b)}\neq\kappa_{(P,b)}^p$. Also, if $\phi$ is an automorphism of $D$ that induces an automorphism of $\mathcal{F}_{(D,b)}(G,b)$ then $\phi|_P$ is an $\mathcal{F}_{(D,b)}(G,b)$-automorphism of $P$ and hence there exists an element $g\in N_G(P,b)$ such that $\phi|_P=c_g$. Therefore the block $b$ satisfies the assumptions of Theorem \ref{thm:three}. In particular, $\OO G\sigma(b)$ and $\OO Gb$ are not $p$-permutation equivalent.
\end{example}


\bibliographystyle{plain}
\bibliography{../../../../Bibliography/bibliography}

\end{document}